\documentclass[a4paper,12pt]{amsart}
\usepackage[T1]{fontenc}
\usepackage[latin1]{inputenc}
\usepackage{fullpage}
\usepackage{times}
\usepackage{graphicx}
\usepackage[colorlinks=true]{hyperref}
\usepackage{color}
\usepackage{amsmath,amsfonts}
\numberwithin{equation}{section}
\usepackage{amsmath, amsthm}
\usepackage{hyperref}
    \usepackage{dsfont}
    \usepackage{fancybox}
    \usepackage{amssymb}

\newcommand{\Rd}{\R^{d-1} \times \left\{0 \right\}}
\newcommand{\R}{\mathbb{R}}

\newcommand{\disp}[1]{\displaystyle{#1}}

\newcommand{\1}{\mathds{1}}

 \newcommand{\spane}{\mathrm{span}}

\newcommand{\lr}[1]{\langle #1 \rangle}

\newcommand{\GG}{\mathcal{G}}
\newcommand{\RR}{\mathcal{R}}

\newcommand{\CC}{\mathcal{C}}

\newcommand{\Es}{E^\ast}
\newcommand{\fs}{f^\ast}

\newcommand{\fns}{f_n^\ast}
\newcommand{\TT}{\mathcal{T}}
\newcommand{\Ss}{\mathcal{S}}
\newcommand{\BB}{\mathcal{B}}
\newcommand{\BBt}{\tilde{\mathcal{B}}}
\newcommand{\II}{\mathcal{I}}
\newcommand{\RRt}{\tilde{\mathcal{\RR}}}
\newcommand{\MM}{\mathcal{M}}
\newcommand{\HHH}{\mathcal{H}}

\title{BEST CONSTANT AND VALUE OF EXTREMIZERS FOR A $k$-PLANE TRANSFORM INEQUALITY}
\date{\today}
\author{ALEXIS DROUOT}

\newtheorem{thm}{Theorem}[section]

\newtheorem{lem}[thm]{Lemma}
\newtheorem{proposition}[thm]{Proposition}
\newtheorem{definition}[thm]{Definition}

\newtheorem{theorem}[thm]{Theorem}

\begin{document}
\hypersetup{
bookmarks=true,
bookmarksopen=true
}

\begin{abstract}
The $k$-plane transform $\RR_k$ acting on test functions on $\R^d$ satisfies a dilation-invariant $L^p \rightarrow L^q$ inequality for some exponents $p,q$. We will explicit some extremizers and the value of the best constant for any value of $k$ and $d$, solving the endpoint case of a conjecture from Baernstein and Loss. This extends their own result for $k=2$ and Christ's result for $k=d-1$.
\end{abstract}

\maketitle
\section{Introduction}

Let us choose $d \geq 2$, $1 \leq k \leq d-1$ and denote by $\GG_k$ the set of all $k$-planes in $\R^d$, that means affine subspaces in $\R^d$ with dimension $k$. We define the $k$-plane transform of a continuous function with compact support $f: \R^d \rightarrow \R$ as 
\begin{equation*}
\RR_kf\left(\Pi\right) = \int_{\Pi} f d\lambda_\Pi
\end{equation*}
where $\Pi \in \GG_k$ and the measure $\lambda_\Pi$ is the surface Lebesgue measure on $\Pi$. The operator $\RR_k$ is known as the Radon transform for $k=d-1$ and as the X-ray transform for $k=1$. It is well known since the works of Oberlin and Stein \cite{oberlin map}, Drury \cite{drury} and Christ \cite{lorentz} that $\RR_k$ can be extended from $L^\frac{d+1}{k+1}\left(\R^d\right)$ to $L^{d+1}\left(\GG_k\right)$ for a certain measure on $\GG_k$ that needs to be defined. Let us denote by $\MM_k$ the submanifold of $\GG_k$ of all $k$-planes containing $0$. The Lebesgue measure on $\R^d$ induces a natural measure on $\MM_k$: there exists a unique probability measure $\mu_k$ on $\MM_k$ invariant in the following sense: if $\Omega$ is an orthogonal map, $P$ is a subset of $\MM_k$, then $\mu_k\left(P\right) = \mu_k\left(\Omega P\right)$. The construction of this measure can be found in \cite{construct}. This induces a measure on $\GG_k$, $\sigma_k$, defined as follows:
\begin{equation}\label{measure}
\sigma_k\left(A\right) = \disp{\int_{\Pi \in \MM_k} \lambda \left( \left\{ x \in \Pi^\perp, x + \Pi \in A \right\} \right) d\mu_k\left(\Pi\right)},
\end{equation}
where $\lambda$ designs the Lebesgue surface measure on the $d-k$-plane. \eqref{measure} defines a measure on $\GG_k$ invariant under translations and rotations in the following sense: if $\Omega$ is an orthogonal map, $P$ is a subset of $\GG_k$, and $x \in \R^d$, then $\sigma_k\left(P\right) = \sigma_k\left(\Omega P+x \right)$.\\

The $L^\frac{d+1}{k+1}\left( \R^d \right)$ to  $L^{d+1}\left( \GG_k, \sigma_k \right)$ boundedness of $\RR_k$ leads to the inequality
\begin{equation}\label{map1}
\disp{\| \RR_k f   \|_{L^{d+1}\left( \GG_k, \sigma_k \right)} \leq A\left(k,d\right) \| f \|_{L^\frac{d+1}{k+1}\left( \R^d \right)}},
\end{equation} 
for a certain constant $A\left(k,d\right)$ chosen to be optimal.\\

Some standard questions appear:
\begin{enumerate}
\item[1-] What is the best constant in the above inequality?
\item[2-] What are the extremizers for this inequality?
\item[3-] Is any extremizing sequence relatively compact -modulo the group of symmetries?
\item[4-] What can we say about functions satisfying $\| \RR_k f  \|_{d+1} \geq c \| f \|_\frac{d+1}{k+1}$? \\
\end{enumerate}

Some of the answers are already known for some values of $k$. In \cite{conjec}, Baernstein and Loss solved the first question for the special case $k=2$, and formulated a conjecture about an extremizer value for a larger class of $L^p \rightarrow L^q$ inequalities. Christ solved their conjecture and answered all the above questions with the three papers \cite{qe christ}, \cite{extr christ}, \cite{rdextr}, for the case $k=d-1$.\\

By a quite different approach, we will give in here a proof of Baernstein and Loss' conjecture \textit{for any value of $k,d$} in the inequality \eqref{map1}. Note that this concerns only the endpoint case of their general conjecture. The value of the extremizers provides the explicit value of the best constant in the inequality \eqref{map1}.

\subsection*{Main result} Our main result is the following theorem:

\begin{theorem}\label{extrem}
\begin{enumerate}
\item[$\left(i\right)$] There exist radial, nonincreasing extremizers for \eqref{map1}. Moreover, any extremizing sequence of nonincreasing, radial functions is relatively compact -modulo the group of dilations.
\item[$\left(ii\right)$] Some extremizers for \eqref{map1} are given by
\begin{equation*}
h\left(x\right) = \left[\dfrac{C}{1+\|Lx\|^2} \right]^\frac{k+1}{2}
\end{equation*}
where $L$ is any invertible affine map on $\R^d$, and $C$ is a constant. 
\item[$\left(iii\right)$] The best constant $A\left(k,d\right)$ is equal to
\begin{equation}\label{extremv}
\dfrac{\| \RR_k h \|_{d+1}}{\| h \|_{\frac{d+1}{k+1}}} = \left[2^{k-d} \dfrac{|S^k|^d}{|S^d|^k} \right]^\frac{1}{d+1}
\end{equation}
where $|S^{i-1}|$ denotes the Lebesgue surface measure of the $i-1$-sphere.
\end{enumerate}
\end{theorem}

The concept of extremizing sequence has not been defined yet. We will say that $f_n$ is an extremizing sequence for the inequality \eqref{map1} if for all $n$, $\| f_n \|_\frac{d+1}{k+1} = 1$ and $\| \RR_k f_n \|_{d+1} \rightarrow A\left(k,d\right)$. Thus the second part of $\left(i\right)$ means that if $f_n$ is an extremizing  nonincreasing, radial functions, then there exists a sequence of real numbers called $\lambda_n$ such that the sequence $x \mapsto \lambda_n^\frac{d\left(k+1\right)}{d+1}f_n\left(\lambda_n x\right)$ admits a subsequence converging in $L^\frac{d+1}{k+1}$.\\

This introduces the main difficulty in $\left(i\right)$. Indeed, the group of invertible, affine maps is a noncompact group of symmetry for \eqref{map1}. Thus if we choose an arbitrary extremizing sequence, then in the most general case it will converge weakly to the null function. We have to overcome this difficulty.\\

Our proof takes its inspiration from three different papers. To prove $\left( i \right)$, we follow Lieb's approach to prove the existence of extremizers for the Hardy-Littlewood-Sobolev inequality, in his famous paper \cite{lieb}. The major difference here is the way to prove that after suitable rescaling, an extremizing sequence converges weakly to a non-zero function. It is not very surprising that Lieb's approach for the Hardy-Littlewood-Sobolev inequality can be used to solve our problem; in fact, it is very similar to the $k$-plane transform inequality. Part $\left(i\right)$ of \ref{extrem} can be seen as a corollary of the following generalized theorem, whose assumptions are also satisfied by the Hardy-Littlewood-Sobolev inequality:
 
\begin{theorem}\label{generali}
Let $m$ be an integer, $\sigma$ a measure on $\R^+$ such that $\sigma \left( \left\{ 0 \right\} \right)=0$. Let $\TT$ be a linear operator satisfying all the below assumptions:
\begin{enumerate}
\item[$\left(i\right)$] $\TT$ maps $L^p\left(\R^+,r^{m-1} dr\right)$ to $L^q\left(\R^+, \sigma\right)$ with $1 < p < q < \infty$ and $\TT$ maps the Lorentz space $L^{p,p+\delta}\left(\R^+,r^{m-1} dr\right)$ to the Lorentz space $L^{q,q-\delta}\left(\R^+,\sigma\right)$ for some $\delta >0$ such that $p+\delta < q-\delta$;
\item[$\left(ii\right)$] $\TT$ satisfies the rearrangement inequality $\| \TT f \|_q \leq \| \TT \left(\fs\right) \|_q$, where $\fs$ denotes the -radial- nonincreasing rearrangement of $f$, with respect to $r^{m-1}dr$;
\item[$\left(iii\right)$] For any nonnegative, nonincreasing function $f$, $\TT f$ is also nonincreasing;
\item[$\left(iv\right)$] The inequality
\begin{equation}\label{mapg}
\| \TT f \|_q \leq A \| f \|_p
\end{equation}
is invariant under the standard action of dilations.
\end{enumerate}
Then the inequality \eqref{mapg} admits nonincreasing, radial extremizers. Moreover, any extremizing sequence of decreasing functions is relatively compact modulo the group of dilations.
\end{theorem}

Two assumptions are essential in this theorem.  The continuity in Lorentz spaces will help us prove a concentration compactness lemma, reducing the difficulties generated by the dilation-invariance. The rearrangement inequality $\left( ii \right)$ will generate some additional compactness.\\

As already said, we can also apply this general theorem to the Hardy-Littlewood-Sobolev inequality. It shows that there exist extremizers. Nevertheless this will not give the value of the best constant, which, furthermore, has been known for a long time. But sometimes the single knowledge of existence of extremizers is enough to get their values and the value of the best constant, because extremizers satisfy a certain Euler-Lagrange equation. If the solutions of this equation are known, then the best constant can be computed.\\

To find the best constant in the $k$-plane inequality \eqref{map1} we will use an approach introduced by Carlen and Loss in \cite{compet}, that they call competing symmetries. We will need the existence of an additional symmetry $\Ss$ of \eqref{map1}, that does \textit{not} map radial functions to radial functions. It could be seen as a problem but it is actually a very helpful information. The choice of this symmetry is the generalisation of a symmetry found by Christ in \cite{rdextr} in the special case of the Radon transform. Coming back to the process introduced by Carlen and Loss and using the existence theorem $\left(i\right)$ it will lead to the value of some extremizers and then of the best constant.\\

Nevertheless the approach that they followed led them to \textit{all} the extremizers, using some additional work for the equality case in the rearrangement inequality. This does not work for us, and then we do not prove that the extremizers are unique modulo the invertible affine maps. In the last section, we explain how a theorem that states that any extremizer can be written $f \circ L$ for $f$ radial and $L$ an invertible affine map actually leads to the explicit value of all the extremizers. A theorem like this one has already been proved by Taryn Flock for $k=1$; it follows that in the case of the X-ray transform, all the extremizers are given by \eqref{extremv}.\\
 
For the rest of the paper, let us note the following:
\begin{itemize}
\item  Let $A$ and $B$ be positive functions and $P$ be some statement. We will say that $P$ implies that $A \lesssim B$ when there exists a -large- universal constant $C$, which depends only on the dimension $d$, such that $P$ implies that $A \leq C B$. $A \gtrsim B$ will be the convert and $A \sim B$ will be used when $A \lesssim B$ and $B \lesssim A$.
\item A radial function will be considered all along the paper either as a function on $\R^d$ or as a function of the norm, depending on the context.
\item $|E|$ denotes the Lebesgue measure of a set $E$, except in the case of a sphere.
\item $d\left(0, \Pi\right)$ denotes the euclidean distance between $0$ and $\Pi$ a $k$-plane, that is
\begin{equation*}
d\left(0,\Pi\right) = \disp{\inf_{y \in \Pi} \| y \|}.
\end{equation*}
\item $|S^{m-1}|$ denotes the Lebesgue surface measure of the euclidean sphere of $\R^m$.
\item $e_d$ is the vector $\left(0,...,0,1\right)$.
\item For $x$ a vector on $\R^d$, we will write $x=\left(x',x_d\right)$ with $x' \in \R^{d-1}$ and $x_d\in \R$.
\item $\| f \|_p$ denotes the $L^p$-norm of $f$, with respect to a contextual measure.
\item $\R^+$ is the set $\left(0,\infty\right)$.\\
\end{itemize} 

I am indebted to Michael Christ who showed me this very interesting subject, and who pointed out some useful papers. I am also indebted to Jean-Marc Delort for a partial proofreading of the manuscript.

\section{Preliminaries}

In this section we introduce some -standards- notions which will be useful for what follows. We will talk about the theory of radial, nonincreasing rearrangement of a function and the theory of Lorentz spaces. Grafakos's book \cite{louk}, is surely a more complete introduction.

\subsection*{Radial nonincreasing rearrangement}\label{rea} Let us consider $\mu$ a measure on $\R^d$ and $E$ a measurable subset of $\R^d$. We denote by $\Es$ the unique closed ball centred at the origin such that $\mu\left(E^\ast\right)=\mu\left(E\right)$; now for $f$ a measurable function from $\R^d$ to $\left[0,\infty\right]$, and $t \geq 0$, let us call 
\begin{equation*}
E_f\left(t\right) = \left\{ x \in \R^d, |f\left(x\right)| \geq t  \right\}.
\end{equation*}
Then we have the following proposition:

\begin{proposition}
Let $f$ be a measurable function from $\R^d$ to $\R \cup {\pm \infty}$; there exists a unique function $\fs$, from $\R^d$ to $\left[0,\infty\right]$  such that
\begin{equation}\label{carac}
E_{|f|}\left(t\right)^\ast = E_{\fs}\left(t\right).
\end{equation}
Moreover, $\fs$ is radial, nonincreasing -as a function of the norm. We furthermore have the properties:
\begin{enumerate}
\item[$\left(i\right)$] for all measurable functions $f \in L^p$, with $1 \leq p \leq \infty$, $\| f  \|_p = \|f^\ast \|_p$,
\item[$\left(ii\right)$] for all measurable nonnegative functions $f \in L^p$, $g \in L^p$, with $1 \leq p \leq \infty$, $\| f - g \|_p \leq \|f^\ast - g^\ast \|_p$,
\item[$\left(iii\right)$] for all measurable nonnegative functions $f$, $g$, $f \leq g \Rightarrow f^\ast \leq g^\ast$,
\item[$\left(iv\right)$] for all measurable nonnegative functions $f$, for all $\lambda \geq 0$, $\lambda f = \left(\lambda f\right)^\ast$.
\end{enumerate}
\end{proposition}

Points $\left(i\right)$ to $\left(iv\right)$ show that the nonlinear operator $f \mapsto f^\ast$ is actually a properly contractive operator -see section \ref{s1}. The map $f^\ast$ is called the symmetric rearrangement of $f$ -with respect to the measure $\mu$. 

\subsection*{Lorentz spaces} Let $1 \leq p < \infty$, $1 \leq r \leq \infty$, $\mu$ a measure on a measurable space $X$. We will call $L^{p,r} = L^{p,r}\left(X,\mu\right)$ the Lorentz space of order $\left(p,r\right)$. Let us recall some notions about Lorentz spaces.
Let $f$ be a function from $X$ to $\R$ and $df$ be the distribution function of $f$, defined as
\begin{equation*}
df\left(t\right) = \mu\left(\left\{ x \in X, |f\left(x\right)| \geq t \right\} \right),
\end{equation*}
for $t \geq 0$. Let us then define the quasi-norm on $L^{p,r}$ as
\begin{equation*}
\| f \|_{L^{p,r}} = \left(\disp{\int_0^\infty \left( df\left(t\right)^\frac{1}{p} t \right)^r \dfrac{dt}{t}}\right)^\frac{1}{r}
\end{equation*}
where the integral is as usual changed to a $\sup$ if $r=\infty$:
\begin{equation*}
\| f \|_{L^{p,\infty}} = \disp{\sup_{t>0} df\left(t\right)^\frac{1}{p} t}.
\end{equation*}
The term quasi-norm means that the quantity defined above does not satisfy the triangle inequality, but satisfies instead the following:
\begin{equation*}
\exists C, \forall f,g \in L^{p,r}, \| f+ g \|_{L^{p,r}} \leq C \left( \|f \|_{L^{p,r}} + \| g \|_{L^{p,r}}\right).
\end{equation*}
The spaces $L^{p,r}$ are quasi-complete for the values of $p,r$ described above.\\

The last thing that we need to know about Lorentz spaces is the useful interpolation inequality 
\begin{equation}\label{interpoline}
\| f \|^r_{L^{p,r}} \leq \| f \|^{r-p}_{L^{p,\infty}} \| f \|^p_p
\end{equation}
for $f \in L^p$, $1 \leq p <\infty$. Indeed,
\begin{equation*}
\| f \|^r_{L^{p,r}} = \disp{\int_0^\infty \left(df\left(t\right)^\frac{1}{p} t\right)^r \dfrac{dt}{t} \leq \left( \sup_{t>0} df\left(t\right)t^\frac{1}{p} \right)^{r-p} \int_0^\infty \left(df\left(t\right)^\frac{1}{p} t\right)^p  \dfrac{dt}{t}= \| f \|^{r-p}_{L^{p,\infty}} \| f \|^p_p}.
\end{equation*}

\section{Existence of extremizers}
We have now all the tools to prove part $\left(i\right)$ of \ref{extrem}. To simplify the notations, let us fix $k$ and $d$ and call $q=d+1$, $p=\frac{d+1}{k+1}$, $\RR =\RR _k$, $\GG_k=\GG$, $\sigma_k=\sigma$ and $A=A\left(k,d\right)$. We are then interested in the existence of extremizers for the inequality
\begin{equation}\label{map2}
\| \RR f \|_q \leq A \| f \|_p .
\end{equation}

A naive approach is of course to consider $\left(f_n\right)$ an extremizing sequence for this inequality, meaning $\| f_n \|_p = 1$ and $\| \RR f_n \|_q \rightarrow A$, and to prove that $f_n$ converges strongly. This, as already said, is not possible. Indeed, the inequality \eqref{map2} enjoys a large and non-compact group of symmetries, the invertible affine maps. By that we mean that if $f \in L^p$ and $L$ is an invertible affine map then we have the identity
\begin{equation*}
\dfrac{\| \RR \left(f \circ L\right) \|_q}{\|f \circ L \|_p} = \dfrac{\| \RR f \|_q}{ \| f \|_p}.
\end{equation*}
For a proof, see lemma \ref{affine}. The non-compactness of this group implies in particular that an arbitrary extremizing sequence has no chance to converge -even weakly- in $L^p$ to a non-zero function. We then need to transform an arbitrary extremizing sequence under the action of invertible affine maps to make it converge. This action is defined by
\begin{equation*}
\left(L,f\right) \mapsto \det\left(L\right)^\frac{1}{p} f \circ L
\end{equation*}
and preserves the $L^p$-norm of $f$ \textit{and} the $L^q$-norm of $\RR f$.\\

\subsection*{Some useful facts}The $k$-plane transform satisfies the rearrangement inequality
\begin{equation}\label{inequat1}
\| \RR g \|_q \leq \| \RR \left( g^\ast \right) \|_q.
\end{equation}
Christ proved this in \cite{lorentz}. That way, instead of considering an arbitrary extremizing sequence, we can consider an extremizing sequence of \textit{radial, nonincreasing} functions. It obviously makes the study much easier, passing from functions on $\R^d$ to nonincreasing functions on $\left[0,\infty\right)$. But the group of dilations is still a non-compact group of symmetries for the $k$-plane transform inequality, even restrained to the radial, nonincreasing functions. Thus we still have to deal with the loss of compactness explained above, but since this loss is only due to dilations, it is easier to deal with.\\

In \cite{lorentz}, Christ proved a really useful boundedness theorem for our purpose: the $k$-plane transform maps the Lorentz space $L^{p,q}$ to the Lebesgue space $L^q$. Note that $p<q$ and then we can apply the interpolation theory for Lorentz spaces -see \cite{interpolation} for instance. It shows that $\RR $ is actually continuous from $L^{p,p+\delta}$ to $L^{q,q-\delta}$ for a certain $\delta > 0$, satisfying $p + \delta < q- \delta$. \\

The following lemma explains how strong assumptions on an extremizing sequence would imply its convergence to an extremizer.

\begin{lem}\label{lemmmma}
Let $X$, $Y$ be two measurable spaces and $\TT$ be a bounded linear operator from $L^p\left(X\right)$ to $L^q\left(Y\right)$, with $1 \leq p \leq q < \infty$. Let us consider $g_n$ an extremizing sequence associated to the inequality $\| \TT f \|_q \leq A \| f \|_p$. Let us assume the three following points:
\begin{enumerate}
\item[$\left(a\right)$] $f_n$ converges weakly to a non-zero function $f \in L^p\left(X\right)$,
\item[$\left(b\right)$] $f_n$ converges almost everywhere to $f$,
\item[$\left(c\right)$] $\TT f_n$ converges almost everywhere to $\TT f$. 
\end{enumerate}
Then we have the following conclusions:
\begin{enumerate}
\item[$\left(i\right)$] $f$ is an extremizer for the above inequality.
\item[$\left(ii\right)$] Actually, $f_n$ converges strongly to $f$.
\end{enumerate}
\end{lem}

This lemma is extremely general and its proof is rather simple; the reader can for instance consult \cite{lieb}. Originally, it was used to prove the existence of extremizers for the Hardy-Littlewood-Sobolev inequality. The three assumptions $\left(a\right), \left(b\right), \left(c\right)$ are very strong. Indeed $\left(a\right)$ can seem easy to be satisfied but when we have a non-compact group of symmetries, as in the inequality \eqref{map2}, an arbitrary extremizing sequence probably converges weakly to $0$. Thus we can do nothing without a concentration-compactness lemma. $\left(b\right)$ requires a certain structure about the extremizing sequence $f_n$. $\left(c\right)$ may be the easiest assumption to show -in particular for integral operator- using that $f_n$ converges weakly. Here we are looking at some radial nonincreasing functions, which makes the study far easier. Indeed, we have the following theorem, which is sometimes called Helly's principle:

\begin{thm}
Let $f_n$ be a sequence of decreasing functions on an interval $I \subset \R$, uniformly bounded. Then up to a passage to a subsequence, $f_n$ converges pointwise.
\end{thm}

This theorem has been known for a while. The idea is basically to extract convergent sequences for all rational points, which leads to a pointwise limit which is decreasing, defined on the rational numbers. Then since the set of points of discontinuity for this limit is countable, we can extract once more and we get a pointwise limit \textit{everywhere}. It then gives a very important compactness result for our purpose. \\

Let us note $d\mu = r^{d-1}dr$. From now, we will consider that $\TT$ is a linear operator and $\sigma$ is a measure such that $\sigma\left( \left\{ 0 \right\} \right)=0$, satisfying the assumptions below:
\begin{enumerate}
\item[$\left(i\right)$] $\TT$ maps $L^p\left(\R^+,\mu\right)$ to $L^q\left(\R^+,\sigma\right)$ with constant $A$ and $1 < p < q < \infty$ and $\TT$ maps $L^{p,p+\delta}\left(\R^+,\mu\right)$ to $L^{q,q-\delta}\left(\R^+,\sigma\right)$ with constant $B$, for a $\delta >0$ such that $p+\delta < q-\delta$;
\item[$\left(ii\right)$] $\TT$ satisfies the rearrangement inequality $\| \TT f \|_q \leq \| \TT \left( \fs \right) \|_q$, where $f^\ast$ is the nonincreasing -radial- rearrangement of $f$ with respect to $\mu$;
\item[$\left(iii\right)$] For any nonincreasing function $f$, $\TT f$ is also nonincreasing;
\item[$\left(iv\right)$] The $L^p \rightarrow L^q$ boundedness inequality is invariant under the action of dilations;
\end{enumerate}
which places us in the general frame of theorem \ref{generali}.\\

The $k$-plane transform does \textit{not} satisfy these assumptions. But because of the rearrangement inequality \eqref{inequat1}, what we need to do is to look for extremizers for $\RR$ restricted to radial functions. On this subset of $L^p$, $\RR$ is closely related to an operator acting on functions on $\R^+$. The geometric point of view make us introduce $\TT$ the operator defined on continuous, compactly supported functions on $\R^+$ as
\begin{equation*}
\TT f\left(r\right) = \disp{\int_0^\infty f \left(\sqrt{s^2 + r^2} \right) s^{k-1}ds} .
\end{equation*}
Then we have the following:

\begin{lem}
For all $f$ radial, continuous, compactly supported function on $\R^d$, and $\Pi \in \GG$ such that $d\left(0,\Pi\right) = r$,
\begin{equation}\label{RT}
\RR f \left(\Pi\right) = |S^{k-1}| \cdot \TT f \left(r\right).
\end{equation}
\end{lem}

\begin{proof} Let us call $P$ the $k$-plane $\R^k \times \left\{ 0 \right\}^{d-k}$.
Let $\Pi \in \GG$ such that $d\left(0,\Pi \right) = r$, and $\Omega$ an isometry of $\R^d$ such that 
\begin{equation*}
\Omega\Pi = \R^k \times \left\{0\right\} \times ... \times \left\{0\right\} \times \left\{r\right\} = re_d + P.
\end{equation*}
Then we know that for a radial function $f$,
\begin{equation*}
\RR f \left(\Pi\right) = \disp{\int_\Pi f\left(x\right) d\lambda_{\Pi} \left(x\right) = \int_{\Omega \Pi} f\left(\Omega^{-1} x\right) d\lambda_{\Pi}\left(\Omega^{-1}x\right)=\int_P f\left(x\right) d\lambda_P\left(x\right). } 
\end{equation*}
The measure on $P$ is as simple as possible, this the Lebesgue measure on $\R^k$. Thus using polar coordinates $\left(s,\theta\right) \in \R^+ \times S^{k-1}$, we get
\begin{equation*}
\RR f \left(\Pi\right) = \disp{\int_{s=0}^\infty \int_{\theta \in S^{k-1}} f\left(re_d+s\theta\right) d\theta s^{k-1} ds    }.
\end{equation*}
Using that $f$ is radial and that $re_d$ and $s\theta$ are orthogonal, we finally get
\begin{equation*}
\RR f \left(\Pi\right) = \disp{ |S^{k-1}| \int_{s=0}^\infty    f \left(\sqrt{s^2 + r^2} \right) s^{k-1}ds = |S^{k-1}| \cdot \TT f \left(r\right).} 
\end{equation*}\end{proof}

 The equation \eqref{RT} shows that $\TT$ is \textit{almost} the $k$-plane transform. $\TT$ acts on some Lebesgue spaces, that we need to explicit, using this correspondence. Its domain is of course the space $L^p\left(\R^+,r^{d-1} dr\right)$. On the other hand, we have
\begin{equation*}
\disp{\| \RR f \|_q^q = \int_{\GG} |\RR f\left(\Pi\right)|^q d\sigma\left(\Pi\right)} 
 = \disp{|S^{k-1}|^q |S^{d-k-1}| \int_{r=0}^\infty |\TT f\left(r\right)|^q r^{d-k-1} dr},
\end{equation*}
where the last line is obtained thanks to the formula $\left(1.1\right)$ in \cite{conjec}. This shows that $\TT$ maps $L^p\left(\R^+, r^{d-1}dr\right)$ to $L^q\left(\R^+, r^{d-k-1}dr\right)$. Using what is written above about the $k$-plane transform, and the same considerations, $\TT$ satisfy all the assumptions of theorem \ref{generali}. The correspondence formula \eqref{RT} finally shows that the inequality $\| \RR f \|_q \lesssim \| f \|_p$ admits extremizers if and only if the inequality $\| \TT f \|_q \lesssim \| f \|_p$ does so. Thus part $\left(i\right)$ of theorem \ref{extrem} is indeed a particular case of theorem \ref{generali}. At last, and this will be useful in the computation of the best constant, for all radial function $f$,
\begin{equation}\label{RTnorm}
\dfrac{\| \RR f \|_{L^q\left( \GG, d\sigma \right)}}{\| f\|_{L^p\left( \R^d \right)}} = \dfrac{|S^{k-1}| |S^{d-k-1}|^\frac{1}{q}}{|S^{d-1}|^\frac{1}{p}} \cdot \dfrac{\| \TT f \|_{L^q\left( \R^+, r^{d-k-1}dr \right)}}{\| f\|_{L^p\left( \R^+, r^{d-1} dr \right)}}.
\end{equation}

Now we have all the tools we need to prove the existence result.
\subsection*{Concentration-compactness result} As we already said, the main difficulty to be overcome is a compactness default for an arbitrary extremizing sequence. The first thing that we need is a way to concentrate some weight inside a bounded domain. The following lemma, which is a form of concentration-compactness principle, is the main idea for the existence theorem:
 
\begin{lem}\label{lem1} There exists a constant $c$ depending only on $p,q,d,\delta$  -that means, only on the parameters- such that the following is satisfied.
Let $f$ be nonincreasing, with $\|f\|_p = 1$ and $\| \TT f \|_q \geq \frac{A}{2}$. There exists $t_0$ such that if we call $g: x \mapsto t_0 f\left(t_0^{\frac{p}{d}}x\right)$, then $g \geq \1_{B\left(0,c\right)}$, $\|g\|_p = 1$ and $\| \TT g \|_q = \| \TT f \|_q$.
 \end{lem}

\begin{proof}
Let us choose $f \in L^p$ nonincreasing, $\|f\|_p = 1$ such that $\| \TT f \|_q \geq \frac{A}{2}$. Then
\begin{align}
\label{equat1} \frac{A}{2} \leq \| \TT f \|_q \lesssim \| \TT f \|_{L^{q,q-\delta}} & \lesssim B \| f \|_{L^{p,p+\delta}}\\
\label{equat2}                  &\lesssim B \| f \|_{L^{p,\infty}}^\frac{\delta}{p+\delta} \| f \|_p^\frac{p}{\delta+p}.
\end{align}
In \eqref{equat1} we used the injection $L^{q,q-\delta} \hookrightarrow L^q$, and the boundedness of the operator $\TT $ from $L^{p,p+\delta}$ to $L^{q,q-\delta}$, with norm that we called $B$. \eqref{equat2} is a consequence of the interpolation inequality \eqref{interpoline}. It leads to: 
\begin{equation}\label{inequa1}
\| f \|_{L^{p,\infty}} \gtrsim C
\end{equation}
with $C$ depending only on $p,q,d,\delta$. 
We thus know that there exists a real number $s_0=t_0^{-1}>0$ such that $s_0 df\left(s_0\right)^\frac{1}{p} \gtrsim C$. Let us call $g$ the function defined by $g\left(x\right) = t_0 f\left(t_0^\frac{p}{d}x\right)$. Then $g$ remains nonincreasing; $\| g \|_p=1$, $\| \TT g \|_q=\| \TT f \|_q$; and 
\begin{equation*}
\mu\left(\left\{ x, g\left(x\right) \geq 1 \right\}\right) = s_0^p \mu\left(\left\{ x, f\left(x\right) \geq s_0 \right\}\right) \gtrsim C.
\end{equation*}
Using now that $g$ is nonincreasing, there exists $c$ depending only on $p,q,d,\delta$ such that $g \geq \1_{B\left(0,c\right)}$. The other consequences follow from the dilation-invariance of the inequality.
\end{proof}

This lemma removes the difficulties generated by the non-compactness of the dilation group. Indeed if we consider an extremizing sequence of nonincreasing functions then it shows that modulo the dilation group there exists a subsequence that converges weakly to a non-zero function. The crucial point here was to use the boundedness in Lorentz space to concentrate most of the $L^p$-norm of $g$ inside a ball with controlled radius, centred at $0$.

\subsection*{Existence of extremizer} In this section, we will see how the previous lemma closes the existence problem.
\begin{proof}
Let $f_n$ be an extremizing sequence for the inequality $\| \TT f \|_q \leq A \| f \|_p$; we can assume that $\| \TT f_n \|_q \geq \frac{A}{2}$. Let us call $\fns$ the nonincreasing rearrangement of $f_n$. Now using lemma \ref{lem1} and inequality \eqref{inequat1} we know that for each $n$ there exists $M_n$ such that if $g_n\left(x\right) = M_n \fns\left(M_n^\frac{p}{d} x\right)$ then $g_n$ is greater than $ \1_{B\left(0,c\right)}$. Here $c$ does not depend on $n$. Moreover, $g_n$ remains nonincreasing, and its $L^p$-norm is still $1$.\\

Using that our inequality is dilation invariant -assumption $\left( iv \right)$- $g_n$ remains an extremizing sequence of nonincreasing functions. Then up to passage to a subsequence, $g_n$ converges $d\mu$-almost everywhere. Indeed let us abuse notations and consider $g_n$ on $\left(0, +\infty\right)$. Let $\rho > 0$; if $g_n\left(\rho\right)$ was an unbounded sequence, then we would be able to extract a subsequence of $g_n$, called $g_{\Phi\left(n\right)}$, such that $g_{\Phi\left(n\right)}\left(\rho\right)$ converges to infinity. Then using that $g_n$ is decreasing and $\rho$ is positive we cannot have $\| g_n \|_p = 1$. That way $g_n$ is uniformly bounded on $\left[\rho, + \infty\right)$. By Helly's theorem $g_n$ converges on $\left[\rho,\infty\right)$, with an extraction. Doing that for a sequence  $\rho_k > 0$ converging to $0$ we get that $g_n$ converges pointwise on $\R^+$, up to an extraction. Now using that $\mu\left( \left\{ 0 \right\} \right) = 0$, the sequence $g_n$ converges $d\mu$-almost everywhere. The extracted sequence will still be called $g_n$. Let us call $g$ the pointwise limit, which satisfies $g \geq \1_{B\left(0,c\right)}$. Thus $g$ is non-zero. \\

$g$ can also be regarded as the weak limit of $g_n$ in $L^p$ -since $g_n$ is bounded in $L^p$ for a value of $p$ greater than $1$, and since the limit of $g_n$ is unique in the distribution space $D'$. It proves that $g$ lies in $L^p$. So far we have proved points $\left(a\right)$ and $\left(b\right)$ in \ref{lemmmma}. Point $\left(c\right)$ is the same as point $\left(a\right)$: since $g_n$ is a sequence of nonincreasing functions $\TT g_n$ is a sequence of nonincreasing functions. Moreover, if there existed a $\rho > 0$ such that $\TT g_n\left(\rho\right)$ were an unbounded sequence, then for the same reason as above the sequence $\|\TT g_n\|_q$ would be unbounded, which is impossible. Then -up to an extraction- $\TT g_n$ must converge everywhere, except maybe at $0$. Using that $\TT $ is linear, continuous from $L^p$ to $L^q$ we know that $\TT $ is continuous from $L^p$ with its weak topology to $L^q$ with its weak topology, and then the pointwise limit of $\TT g_n$ is $\TT g$.\\

Finally, using the inequality $p \leq q$ we can apply \ref{lemmmma} and we get the part $\left(i\right)$ of theorem \ref{extrem}.
 
\end{proof}

\subsection*{Application to the sharp Hardy-Littlewood-Sobolev inequality.} Here we will apply theorem \ref{generali} to the widely studied inequality:
\begin{equation*}
\| \|x\|^{-\lambda} \ast f \|_q \leq A\left(\lambda,p\right) \| f\|_p
\end{equation*}
with $p, \lambda$ and $q$ related through
\begin{equation*}
\dfrac{1}{p}+\dfrac{\lambda}{d} = 1 + \dfrac{1}{q};
\end{equation*}
\begin{equation*}
1<p<q<\infty;
\end{equation*}
\begin{equation*}
0<\lambda<d.
\end{equation*}
This is the Hardy-Littlewood-Sobolev inequality. It admits extremizers -see \cite{compet}, \cite{lieb}. Our theorem can be directly applied here. Indeed  let $\TT$ be the operator defined as
\begin{equation*}
\TT: f \mapsto g \ast f 
\end{equation*}
where $g$ is the function defined as $g\left(x\right)=\| x \|^{-\lambda}$, $0 < \lambda < d$. It is important to note that $g$ lies in the Lorentz space $L^{\frac{d}{\lambda},\infty}$. Let us check the assumptions of the above theorem:
\begin{enumerate}
\item[$\left(i\right)$] The operator $\TT$ is continuous from $L^p\left(\R^d\right)$ to $L^q\left(\R^d\right)$ with $1<p < q<\infty$ satisfying
\begin{equation*}
\dfrac{1}{p}+\dfrac{\lambda}{d} = 1 + \dfrac{1}{q}.
\end{equation*}
$\TT$ is more generally continuous from $L^{p,r}$ to $L^{q,s}$ for all $s \geq r$ -see O'Neil, \cite{convol}. We can then choose $\delta$ such that $p < r=p+\delta < s=q - \delta < q$.
\item[$\left(ii\right)$] $\TT$ satisfies the Riesz rearrangement inequality $\| \TT \left(f^\ast\right) \|_q \geq \| \TT f \|_q$, since $g$ satisfies $g^\ast = g$.
\item[$\left(iii\right)$] For any nonnegative, nonincreasing, radial function $f$, $\TT f$ is also radial, nonincreasing -see \cite{lieb} for instance.
\item[$\left(iv\right)$] Using that the function $g$ is homogeneous the inequality $\| \TT f \|_q \leq A \| f \|_p$ is dilation-invariant.
\end{enumerate}
Thus we can apply the theorem that we just proved, restraining $\TT$ to radial functions regarded as functions of the norm. It tells us that the Hardy-Littlewood-Sobolev inequality $\| \TT f \|_q \leq A \| f \|_p$ admits extremizers. This is a well known result but we believe that the way to prove it, especially the concentration-compactness lemma, is new. It is important to note that none of the assumptions above were hard to prove.

\section{Best constant and value of extremizers for the $k$-plane inequality} So far we have proved a general existence theorem. Applied to the $k$-plane transform inequality \eqref{map1}, it leads to the existence of extremizers. We will give here the value of some extremizers and of the best constant, which solves the endpoint case of Baernstein and Loss conjecture in \cite{conjec}.\\

 The existence of a large group of symmetries was clearly an obstacle to overcome to prove the existence of the extremizers. We will see in this section that this is no longer an obstacle for the research of the explicit values of extremizers, but rather an aid: we will even look for additional symmetries. \\
 
We start this section by a small lemma that is needed to get the explicit value of extremizers:

\begin{lem}
If $f$ is an extremizer for \eqref{map1} then $f$ does not change its sign.
\end{lem}

\begin{proof}If $f$ is an extremizer, then using $A=\| \RR f \|_q \leq \| \RR \left(|f|\right) \|_q\leq A$ we deduce that $|\RR f|=\RR \left(|f|\right)$ almost everywhere on $\GG$. Then we can assume that $\RR f \geq 0$ -almost everywhere on $\GG$. Let us call $E=\left\{x, f\left(x\right) \geq 0\right\}$. Then using that $f\1_E \geq f$, $\RR \left(f\1_E\right) \geq \RR f$ and $\|f \1_E \|_p \leq \| f\|_p$, $f\1_E$ must be an extremizer. Using that $f$ is an extremizer too, $\|f \1_E \|_p = \| f\|_p$ and so $|E^c|=0$.\end{proof}

Thus we can consider extremizers that are nonnegative. Here we want to prove the following:

\begin{theorem}\label{sharp} An extremizer for the inequality \eqref{map1} is given by
\begin{equation}\label{extremi}
f\left(x\right) =\left[ \dfrac{1}{1+\| x\|^2}\right]^\frac{k+1}{2}.
\end{equation}
\end{theorem}

As a matter of fact, since any invertible affine map is a symmetry of the inequality \eqref{map1}, this theorem is equivalent to part $\left(ii\right)$ of theorem \ref{extrem}.\\

Let us explain the process of the proof before the details. Our purpose here is to introduce two operators $V, \Ss$ acting on $L^p$, formally satisfying: $V$ and $\Ss$ preserve the $L^p$-norm of suitable functions, \textit{and}
\begin{equation}\label{ineq132}
\| \RR f \|_q \leq \| \RR \Ss f \|_q; \| \RR f \|_q \leq \| \RR V f \|_q .
\end{equation}
This means that $V$ and $\Ss$ globally increase the functional $f \mapsto \frac{\| \RR f \|_q}{ \|f \|_p}$. Now using additional properties of $\Ss$ and $V$, we will apply a theorem from Carlen and Loss stated in \cite{compet} to show that for any choice of $f \in L^p$, the sequence $\left(V\Ss\right)^n f$ converges to an explicit function $h$. Starting from a function $f$ which is an extremizer, and using \eqref{ineq132}, $h$ must be an extremizer and is explicitly known.\\

In practice, the operator $V$ will be the symmetric rearrangement $f \mapsto f^\ast$, and $\Ss$ will be a symmetry of the inequality. The class of functions whose norm is preserved under the action of $V$ and $\Ss$ will be the nonnegative functions. The operator $\Ss$ is special in a certain sense: it does not preserve the class of radial functions. Thus if we were able to construct an extremizer such that $\Ss h = h$ and $Vh = h$, then the explicit value of $h$ could be determined. A way to construct such an extremizer is described in the next section. But we can already note that an extremizer satisfying this condition must satisfies $\left(V\Ss\right)^n h = h$ for all $n$; this way, considering the sequence $\left(V \Ss\right)^n f$ where $f$ is already an extremizer is probably a good idea.

\subsection*{Competing operators.}\label{s1}

As we said we are following the approach introduced by Carlen and Loss in \cite{compet}. We might as well refer to the book \cite{boook}. In a first time we sum up the general results stated in this book, chapter II, paragraph $3.4$: let $\BB$ be a Banach space of real valued functions, with norm $\| \cdot \|$. Let us consider $\BB^+$ the cone of nonnegative functions; let us assume that $\BB^+$ is closed. Let us introduce some definitions:

\begin{definition}
An operator $A$ on $\BB$ is called properly contractive  provided that
\begin{enumerate}
\item[$\left(i\right)$] $A$ is norm preserving on $\BB^+$, i.e., $\| Af \| = \| f \|$ for all $f \in \BB^+$,
\item[$\left(ii\right)$] $A$ is contractive on $\BB^+$, i.e., for all $f, g \in \BB^+$, $\| Af - Ag \| \leq \| f - g \|$,
\item[$\left(iii\right)$] $A$ is order preserving on $\BB^+$, i.e., for all $f, g \in \BB^+$, $f \leq g \Rightarrow Af \leq Ag$,
\item[$\left(iv\right)$] $A$ is homogeneous of degree one on $\BB^+$, i.e., for all $f \in \BB^+, \lambda \geq 0$, $A\left(\lambda f\right) = \lambda Af$.
\end{enumerate}
\end{definition}

Note that we do not need $A$ to be linear. Some examples of such operators are for instance the radial nonincreasing rearrangement $f \mapsto f^\ast$ or any linear isometry on $\BB$.

\begin{definition}\label{def11}
Given a pair of properly contractive operators $\Ss$ and $V$, it is said that $\Ss$ competes with $V$ if for $f \in \BB^+$,
\begin{equation*}
f \in R\left(V\right) \cap \Ss R\left(V\right) \Rightarrow \Ss f = f.
\end{equation*}
Here $R$ denotes the range.
\end{definition}

\begin{theorem}\label{competing}
Suppose that $\Ss$ and $V$ are both properly contractive, that $V^2=V$ and that $\Ss$ competes with $V$. Suppose further that there is a dense set $\BBt \subset \BB^+$ and sets $K_N$ satisfying $\cup_N K_N = \BBt$ and for all integer $N$, $\Ss K_N \subset K_N $, $VK_N \subset K_N$, and $VK_N$ relatively compact in $\BB$. Finally suppose that there exists a function $h \in \BB^+$ with $\Ss h = Vh = h$ and such that for all $f \in \BB^+$,
\begin{equation}\label{implic11}
\| Vf -h \| = \| f -h \| \Rightarrow Vf = f.
\end{equation}
Then for any $f \in \BB^+$, 
\begin{equation*}
\disp{Tf \equiv \lim_{n \rightarrow \infty} \left(V \Ss\right)^n f}
\end{equation*}
exists. Moreover, $\Ss T = T$ and $VT = T$.
\end{theorem}

\subsection*{An additional symmetry.}

Now we come back to the work of Christ. Using correspondence between a convolution operator that he studied in the three papers \cite{qe christ}, \cite{extr christ}, \cite{smooth christ} he proved in \cite{rdextr} the existence of an additional symmetry for the Radon transform inequality, which is the case $k=d-1$. It is defined as:
\begin{equation*}
\II f \left(u,s\right) = \dfrac{1}{|s|^d} f \left( \dfrac{u}{s} , \dfrac{1}{s} \right).
\end{equation*}
It satisfies then $\| \II f \|_{\frac{d+1}{d}} = \| f \|_\frac{d+1}{d}$ and $\| \RR_{d-1} \II f \|_{d+1} = \| \RR_{d-1} f \|_{d+1}$. Fortunately it happens that this symmetry, slightly modified, is working for the $L^p \rightarrow L^q$ inequality related to the $k$-plane transform.

\begin{lem}\label{l13}
Let $\Ss$ be the operator defined as
\begin{equation*}
\Ss f \left(u,s\right) = \dfrac{1}{|s|^{k+1}} f \left( \dfrac{u}{s} , \dfrac{1}{s} \right)
\end{equation*}
where $\left(u,s\right) \in \R^{d-1} \times \left( \R-\left\{0\right\} \right)$. Then $\Ss$ is an isometry of $L^p$ and satisfies the identity:
\begin{equation}\label{equal11}
\| \RR \Ss f \|_q = \| \RR f \|_q ,
\end{equation}
for any nonnegative function $f$.
\end{lem}

\begin{proof} Let us check first that $\Ss$ is an isometry of $L^p$. Let us call 
\begin{equation*} \Phi\left(x\right) = \left( \dfrac{x'}{x_d}, \dfrac{1}{x_d} \right) \end{equation*} 
for $x=\left(x', x_d\right) \in \R^{d-1} \times \left(\R-\left\{ 0 \right\}\right)$. Then its Jacobian determinant is
\begin{equation*}
J\Phi\left(x\right) = \dfrac{1}{|x_d|^{d+1}},
\end{equation*}
which shows that $\| \Ss f \|_p = \| f \|_p$.
 Then we just have to prove the equality \eqref{equal11}. The proof is nothing more than calculation. Let us introduce a bunch of notations before we begin:
\begin{equation*} \Pi\left(x_0, ..., x_k\right) \end{equation*}
denotes the unique $k$-plane containing the linearly independent points $x_0, ..., x_k \in \R^d \times ... \times \R^d$; Let us define $\RRt f$ as
\begin{equation*} \disp{\RRt f\left(x_0, ..., x_k \right) = \int_{\R^k} f\left(x_0 + \lambda_1 \left(x_1-x_0\right) + ... + \lambda_k \left(x_k-x_0\right)\right) d\lambda_1 ... d\lambda_k}. \end{equation*}
Thus we have the correspondence
\begin{equation}\label{correspR}
V\left(x_0, ..., x_k\right) \cdot \RRt f\left(x_0, ..., x_k \right) = \RR f \left(\Pi \left(x_0, ..., x_k\right)\right)
\end{equation}
where $V\left(x_0, ..., x_k\right)$ is the volume of the $k$-simplex $\left(x_0, ..., x_k\right)$.
We want in a first time to prove the following pointwise estimate:

\begin{lem}\label{l14} For all $f \in C^\infty_0$, for all $x_0, ..., x_k \in \R^d \times ... \times \R^d$, linearly independent and such that $\Phi\left(x_0\right), ..., \Phi\left(x_k\right)$ exist and are linearly independent,
\begin{equation*}
\left(\RRt \Ss f\right) \left(x_0, ..., x_k\right) = \dfrac{\left(\RRt f\right) \left(\Phi\left(x_0\right), ..., \Phi\left(x_k\right)\right)}{|x_{0d} \cdot ... \cdot x_{kd}|} .
\end{equation*}
\end{lem}

\begin{proof} There might be a simpler proof but we can only offer some calculus to state this identity. Let us call $\alpha = \left[x_0 + \lambda_1 \left(x_1-x_0\right) + ... + \lambda_k \left(x_k-x_0\right)\right]_d$ -which implicitly depends on $\lambda_1, ..., \lambda_k$- and $\lambda = \left(\lambda_1, ..., \lambda_k\right) \in \R^k$. Thus
 \begin{equation}\label{equat55}
 \left(\RRt \Ss f\right) \left(x_0, ..., x_k\right) = \disp{ \int_{\R^k} \dfrac{1}{|\alpha|^{k+1}} f\left( \dfrac{x'_0 + \lambda_1 \left(x'_1-x'_0\right) + ... + \lambda_k \left(x'_k-x'_0\right)+e_d} {\alpha} \right)  d\lambda }.
 \end{equation}
 Let us make the change of variable
 \begin{equation}\label{changevar}
\lambda'_1 = \alpha^{-1}\lambda_1; ...; \lambda'_{k-1} = \alpha^{-1}\lambda_{k-1}; \lambda'_k = \alpha^{-1}. 
 \end{equation}
Then
\begin{equation*}
d\lambda' =  \dfrac{|\left[x_k-x_0\right]_d|}{|\alpha|^{k+1}} d\lambda
\end{equation*}
and \eqref{equat55} becomes 
 $$\left(\RRt \Ss f\right) \left(x_0, ..., x_k\right) =$$
 $$\disp{ \int_{\R^k} f \left( y_k + \lambda'_k \left( x'_0 + e_d -x_{0d} y_k \right) + \sum_{i=1}^{k-1}\lambda'_i \left( x'_i-x'_0 - \left[x_i-x_0\right]_d y_k \right) \right) \dfrac{d\lambda'}{|\left[x_k-x_0\right]_d|}}$$
 where
\begin{equation*}
y_i = \dfrac{x'_i-x'_0}{\left[x_i-x_0\right]_d}.
\end{equation*}
This formula is somehow important: it shows that we are still integrating $f$ over a $k$-plane. Which one? When we were computing $\RRt \Ss f\left(x_0, ..., x_k\right)$, we were interested only by the values of $f$ on $\Phi\left(\Pi\left(x_0, ..., x_k\right)\right)$. That way it is simple to guess that $\RRt \Ss f\left(x_0, ..., x_k\right)$ is closely related to $\Pi\left(\Phi\left(x_0\right), ..., \Phi\left(x_k\right)\right)$. And indeed, we just have to check that any of the points $x_j$ can be written
\begin{equation}\label{changevar2}
x_j =y_k + \lambda'_k \left(x'_0 + e_d - y_k\right) + \sum_{i=0}^{k-1}\lambda'_i \left( x'_i-x'_0 - \left[x_i-x_0\right]_dy_k \right)
\end{equation}
for suitable choice of $\lambda'$. Taking $\lambda=e_j$ and $\lambda'$ given by \eqref{changevar} for this choice of $\lambda$, we get the equality \eqref{changevar2}. Let us now make the other change of variables:
\begin{equation*}
\lambda'_1 = \dfrac{\lambda_1}{\left[x_1-x_0\right]_d}, ..., \lambda'_{k-1} = \dfrac{\lambda_{k-1}}{\left[x_{k-1}-x_0\right]_d}, \lambda'_k =\dfrac{\lambda_k}{x_{0d}}.
\end{equation*}
We finally get:
$$\left(\RRt \Ss f\right) \left(x_0, ..., x_k\right) =$$
$$\disp{ \int_{\R^k} f \left( y'_k + \lambda_k \left( \Phi\left(x_0\right)-y'_k \right) + \sum_{i=1}^{k-1}\lambda_i\left( y'_i - y'_k\right) \right) \dfrac{d\lambda}{|x_{0d}|\prod_{i=1}^{k-1} |\left[x_i-x_0\right]_d|}}.$$
Let us come back to the correspondence between $\RR$ and $\RRt$, \eqref{correspR}. Since we want to find a relation between $\left(\RRt \Ss f\right) \left(x_0, ..., x_k\right)$ and $\left(\RRt f\right) \left(\Phi\left(x_0\right), ..., \Phi\left(x_k\right)\right) $, the above algebra tells us that it is equivalent to find a relation between the two following volumes:
\begin{equation*}
 V\left(  \Phi\left(x_0\right), y_1, ..., y_k\right); V\left(\Phi\left(x_0\right), \Phi\left(x_1\right), ..., \Phi\left(x_k\right)\right).
\end{equation*}

\begin{lem}\label{l15}
$V\left(  \Phi\left(x_0\right), y_1, ..., y_k\right)$ and $ V\left(\Phi\left(x_0\right), \Phi\left(x_1\right), ..., \Phi\left(x_k\right)\right)$ are related through
\begin{equation*}
\disp{\dfrac{V\left(\Phi\left(x_0\right), \Phi\left(x_1\right), ..., \Phi\left(x_k\right)\right)}{V\left(  \Phi\left(x_0\right), y_1, ..., y_k\right)} = \prod_{i=1}^k \left| \dfrac{x_{0d}}{x_{id}} -1 \right| }.
\end{equation*}
\end{lem}
\begin{proof}
This is an easy calculation. With a direct calculus,
\begin{equation*}
\dfrac{x_{id}}{\left[x_i-x_0\right]_d} \left[\Phi\left(x_i\right) - \Phi\left(x_0\right)\right] = \dfrac{x_{0d}x'_i+x_{0d}e_d-x_{id} x'_0 - x_{id}e_d}{x_{0d} \left[x_i-x_0\right]_d}
\end{equation*}
and on the other hand,
\begin{equation*}
y_i - \Phi\left(x_0\right)= \dfrac{x_{0d}x'_i+x_{0d}e_d-x_{id} x'_0 - x_{id}e_d}{x_{0d} \left[x_i-x_0\right]_d}.
\end{equation*}
It proves the equality
\begin{equation*}
 \Phi\left(x_i\right) - \Phi\left(x_0\right) = \left( 1- \dfrac{x_{0d}}{x_{id}}\right) \left[y_i - \Phi\left(x_0\right)\right].
\end{equation*}
Thus using that
\begin{equation*}
V\left(\Phi\left(x_0\right), \Phi\left(x_1\right), ..., \Phi\left(x_k\right)\right) = V\left(0,\Phi\left(x_1\right)-\Phi\left(x_0\right), ..., \Phi\left(x_k\right)-\Phi\left(x_0\right)\right)
\end{equation*}
the lemma \ref{l15} is proved.
\end{proof}

Let us come back to the proof of lemma \ref{l14}. Using the correspondence described in \eqref{correspR} and the previous lemma we finally get the equality
\begin{equation*}
\left(\RRt \Ss f\right) \left(x_0, ..., x_k\right) = \dfrac{\left(\RRt f\right) \left(\Phi\left(x_0\right), ..., \Phi\left(x_k\right)\right)}{|x_{0d} \cdot ... \cdot x_{kd}|} .
\end{equation*}\end{proof}

At last, let us come back to the proof of lemma \ref{l13}. Since the set of bad points $x_0, ..., x_k$ -we mean points which do not satisfy the natural assumptions of \ref{l14}- has null Lebesgue measure in $\left(\R^d\right)^k$ we do not consider them. Let us use Drury's formula, proved in \cite{drury}:
\begin{equation}\label{druryf}
\| \RR f \|_q^q = \disp{ \int_{\left(\R^d\right)^k} dx_0 ... dx_k f\left(x_0\right) \cdot ... \cdot f\left(x_k\right) \cdot \RRt f \left(x_0, ..., x_k\right)^{d-k} }. 
\end{equation}
Now everything that remains to be done is an easy change of variable $z_i = \Phi\left(x_i\right)$. Indeed, 
\begin{align*}
\| \RR \Ss f \|_q^q & = \disp{\int_{\left(\R^d\right)^k} dx_0 ... dx_k \dfrac{1}{|x_{0d}|^{k+1}}f\left(\Phi\left(x_0\right)\right) \cdot ... \cdot \dfrac{1}{|x_{kd}|^{k+1}}f\left(\Phi\left(x_k\right)\right) \cdot \left(\RRt \Ss f \left(x_0, ..., x_k\right)\right)^{d-k} }\\
                    & = \disp{\int_{\left(\R^d\right)^k} dx_0 ... dx_k \dfrac{1}{|x_{0d}|^{d+1}}f\left(\Phi\left(x_0\right)\right) \cdot ... \cdot \dfrac{1}{|x_{kd}|^{d+1}}f\left(\Phi\left(x_k\right)\right) \cdot \left(\RRt f \left(\Phi\left(x_0\right), ..., \Phi\left(x_k\right)\right)\right)^{d-k} } \\
                    & = \disp{\int_{\left(\R^d\right)^k} dz_0 ... dz_k  f\left(z_0\right) \cdot ... \cdot  f\left(z_k\right) \cdot \RRt f \left(z_0, ..., z_k\right)^{d-k} } \\
                    & = \| \RR f \|_q^q .
\end{align*} \end{proof}

Since we introduced a lot of material, it is convenient to prove only now what we have claimed all along the paper:

\begin{lem}\label{affine}
Let $f$ lie in $L^p$ and $L$ be an invertible affine map, then
\begin{equation*}
\dfrac{\| \RR \left(f \circ L\right) \|_q}{\|f \circ L \|_p} = \dfrac{\| \RR f \|_q}{ \| f \|_p}.
\end{equation*}
\end{lem}

\begin{proof}
The proof is a direct consequence of the correspondence formula \eqref{correspR} and of Drury's formula \eqref{druryf}. Indeed, let $L$ be an invertible affine map then
\begin{equation*}
\RRt \left(f \circ L\right)\left(x_0,...,x_k\right) = \RRt f \left(Lx_0, ..., Lx_k\right)
\end{equation*}
and then by the change of variable $z_i=Lx_i$ in Drury's formula we get
\begin{equation*}
\| \RR\left(f\circ L\right) \|_q = |\det\left( L \right)|^{-\frac{1}{p}} \| \RR f \|_q, 
\end{equation*}
which ends the proof.
\end{proof}

Our goal is now to apply theorem \ref{competing}. We have two operators acting on $L^p$ which are increasing the $L^q$-norm of the $k$-plane transform, and preserving the norm of nonnegative, $L^p$-functions. Let us call $V$ the rearrangement operator $f \mapsto f^\ast$, we get the following proposition which is almost the end of the proof:

\begin{proposition}
The operators $V$ and $\Ss$ satisfy the assumptions of theorem \ref{competing}, with the Banach space $L^p$.
\end{proposition}

\begin{proof}
$\Ss$ and $V$ are both properly contractive operators. Let us check that $\Ss$ competes with $V$: this is an easy consequence of the below lemma, \ref{fixepoint}. We now have to check that $\Ss$ and $V$ satisfy the assumptions of \ref{competing}. We follow the arguments of Carlen and Loss in \cite{boook}. Let us define 
\begin{equation*}
h\left(x\right) = \left[\dfrac{1}{ 1 + \| x \|^2 } \right]^\frac{k+1}{2}.
\end{equation*}
Then $\Ss h =h$, $Vh = h$, and so with
\begin{equation*}
K_N = \left\{ f \in L^p, 0 \leq f \leq Nh \right\},
\end{equation*}
it is straightforward to check that $VK_N \subset K_N$, $\Ss K_N \subset K_N$. Moreover $VK_N $ is a compact subset of $L^p$, Indeed, let us consider a sequence $f_n \in VK_N$. Then $f_n$ is radial, nonincreasing, and since $h$ lies in $L^\infty$, the sequence $f_n$ is bounded in $L^\infty$. Thus because of Helly's principle $f_n$ admits a subsequence that is converging almost everywhere.  But since $0 \leq f_n \leq Nh$, because of the dominated convergence theorem this subsequence also converges in $L^p$, which implies that $VK_N$ is relatively compact. Moreover, $\tilde{L^p} = \cup_N K_N$ is a dense subset of nonnegative elements of $\tilde{L^p}$ -since nonnegative, continuous, compactly supported functions are dense in $\tilde{L^p}$.\\

The hardest part is to prove the assumption \eqref{implic11}. Fortunately, since $h$ is strictly nonincreasing, it has already been done in \cite{compet}.
\end{proof}

\begin{lem}\label{fixepoint}
Let $h \in L^p$ such that $Vh = \Ss h =h$. Then there exists a constant $C$ such that
\begin{equation*}
h\left(x\right) = C \left[\dfrac{1}{ 1 + \| x \|^2 } \right]^\frac{k+1}{2}.
\end{equation*}
\end{lem}

\begin{proof}
Let us choose $h$ such that $\Ss h=Vh=h$, then $h$ is equal to its own rearrangement and so is defined on -at least- $\R^d-\left\{0\right\}$. Moreover, $\Ss h$ must be radial. This leads to
\begin{equation*}
\Ss h \left(u,\sqrt{1+u^2}\right) = \left[\dfrac{1}{ 1 + u^2 } \right]^\frac{k+1}{2} h \left( \dfrac{u}{\sqrt{1+u^2}}, \dfrac{1}{\sqrt{1+u^2}} \right) =  \left[\dfrac{1}{ 1 + u^2 } \right]^\frac{k+1}{2} h\left(e_d\right)
\end{equation*}
using that $h$ is radial. But since $h=\Ss h$ is also radial,
\begin{equation*}
\Ss h \left(u,\sqrt{1+u^2}\right) = \Ss h\left(0, \sqrt{1+2u^2}\right) = h\left(0, \sqrt{1+2u^2}\right)
\end{equation*}
we get the equality
\begin{equation}\label{eq38}
h\left(x\right) = h\left(0, \|x \|\right) = \left[\dfrac{2}{ 1 + \| x \|^2 } \right]^\frac{k+1}{2} h\left(e_d\right)
\end{equation}
for all $x \in \R^d$ such that $\| x \| \geq 1$. For $\|x \| < 1$ the equality $\Ss h = h$ shows that \eqref{eq38} is also right, which proves the lemma.
\end{proof}

\subsection*{Proof of the theorem}
Now we have all the material that we need to prove \ref{sharp}. Let $f_0 \geq 0$ be an extremizer for \eqref{map1}, whose existence has already been proved. Let us define the limit 
\begin{equation*}
\disp{h= Tf_0 = \lim_{n \rightarrow \infty} \left(V \Ss\right)^n f_0}. 
\end{equation*}
Then because of the inequality \eqref{inequat1} and the equality \eqref{equal11}, $h$ is still an extremizer. Moreover, because of theorem \ref{competing}, $Vh = \Ss h = h$ and then $h$ satisfies the assumptions of lemma \ref{fixepoint}. We then get:
\begin{equation*}
h\left(x\right) = h\left(e_d\right) \left[\dfrac{2}{ 1 + \| x \|^2 } \right]^\frac{k+1}{2}.
\end{equation*}

\subsection*{Value of the best constant} Here we describe the way to compute the value of the best constant. We use the correspondence \eqref{RT} described in the previous section, and the formula \eqref{RTnorm}, and only think about $\TT$ and its related measurable spaces instead of $\RR$. Let $h$ be the radial extremizer
\begin{equation*}
h\left(r\right) = \left[ \dfrac{1}{1+r^2} \right]^\frac{k+1}{2}.
\end{equation*}
A family of integrals will be useful to compute its $L^p$-norm and the $L^q$-norm of $\TT h$. These integrals are defined as
\begin{equation*}
I\left(m,n\right) = \disp{\int_0^\infty} \dfrac{t^m}{\left(1+t^2\right)^\frac{n}{2}} dt .
\end{equation*}
The change of variable $t=\tan\left(\theta\right)$ states a relation between $I\left(n,m\right)$ and the function $\beta$:
\begin{equation*}
I\left(m,n\right) = \dfrac{1}{2} \beta \left(\dfrac{m+1}{2} , \dfrac{n-m-1}{2} \right).
\end{equation*}
We recall that the function $\beta$ is defined as
\begin{equation*}
\beta\left(x,y\right) = \disp{\dfrac{1}{2} \int_{\theta =0}^{\frac{\pi}{2}} \sin\left(\theta\right)^{2x-1} \cos\left(\theta\right)^{2y-1} d\theta }.
\end{equation*}
Then:
\begin{equation*}
\| h \|_p^p = \disp{\int_0^\infty \dfrac{r^{d-1}dr}{\left(1+r^2\right)^\frac{d+1}{2}} = I\left(d-1,d+1\right) = \dfrac{1}{2} \beta \left(\dfrac{d}{2},\dfrac{1}{2}\right) };
\end{equation*}
\begin{equation*}
\TT h \left(r\right) = \dfrac{1}{\sqrt{1+r^2}} \disp{ \int_0^\infty \dfrac{u^{k-1} du}{\left(1+u^2\right)^\frac{k+1}{2}} } = \dfrac{1}{\sqrt{1+r^2}} I\left(k-1,k+1\right)
\end{equation*}
using the change of variable $s^2 = \left(1+r^2\right) u^2$;
\begin{equation*}
\| \TT h \|_q^q = I\left(k-1,k+1\right)^q \disp{ \int_0^\infty \dfrac{r^{d-k-1} dr}{\left(1+r^2\right)^\frac{d+1}{2}}  = I\left(k-1,k+1\right)^q I\left(d-k-1,d+1\right)}.
\end{equation*}
Now let us recall the fundamental relations:
\begin{equation*}
\beta\left(x,y\right) = \dfrac{\Gamma\left(x\right) \Gamma\left(y\right)}{\Gamma\left(x+y\right)};
\end{equation*}
\begin{equation*}
\dfrac{1}{2} |S^{n-1}| \Gamma\left(\dfrac{n}{2}\right) = \pi^\frac{n}{2}.
\end{equation*}
They lead to the formula
\begin{equation*}
A\left(k,d\right) =\dfrac{\| \RR h \|_q}{\| h \|_p} =  \pi^\frac{d-k}{2\left(d+1\right)} \cdot \Gamma\left(\dfrac{d+1}{2}\right)^\frac{k}{d+1} \cdot \Gamma\left(\frac{k+1}{2}\right)^{-\frac{d}{d+1}} = \left[ 2^{k-d}\dfrac{|S^k|^d}{|S^d|^k} \right]^\frac{1}{d+1}.
\end{equation*}
For instance, in the cases of the X-ray and the Radon transform transform in the $3$-dimensional space,
$$A\left(1,3\right) = \pi^\frac{1}{4} \simeq 1,33 $$
$$A\left(2,3\right) = \pi^{-\frac{3}{8}} \simeq 0,651.$$

\subsection*{An alternative proof that does not use the existence theorem} Let us consider an extremizing sequence for \eqref{map1}, called $f_m$. Then $|f_m|$ is still an extremizing sequence and so we can assume that $f_m$ is nonnegative. Then for each integer $m$, the sequence $\left(\Ss V\right)^n f_m$ converges to a function $h_m$ whose value is given by \ref{fixepoint},
 \begin{equation*}
 h_m\left(x\right) =  C_m \left[\dfrac{1}{ 1 + \| x \|^2 } \right]^\frac{k+1}{2}.
 \end{equation*}
Moreover, since $f_m$ is normalized, it is the same for $h_m$ which forces the constant $C_m$ to be independent of $m$, and then $h_m$ to be independent of $m$. At last, 
\begin{equation*}
A \longleftarrow \| \RR f_m \|_q \leq \| \RR h \|_q
\end{equation*}
which proves that $h$ is an extremizer, \textit{without the existence part of theorem \ref{extrem}}.

\section{The question of the uniqueness}

We shall discuss here the question of the uniqueness of extremizers of \eqref{map2}. We will assume $d \geq 3$. This is not annoying: indeed, for the case $d=2$, the only $k$-plane transform is the Radon transform and has been thoroughly  studied by Christ in \cite{rdextr}.\\

The uniqueness problem for the Radon transform has been solved by Christ in his paper \cite{rdextr}. The main tool for the proof is the following:
\begin{theorem}\label{compoL}
Let $k=d-1$, and $f$ be a nonnegative extremizer. Then there exist a radial, nonincreasing, nonnegative extremizer $F$, and an invertible affine map $L$, such that $f=F \circ L$.
\end{theorem}

Then it followed that all the work was almost done. Christ characterized all the extremizers, using the uniqueness theorem \ref{compoL} two times, in a certain sense. His approach is very interesting because the question of the uniqueness is curiously intertwined with the question of the existence. Here we want to develop a different approach, for an arbitrary $0 \leq k \leq d-1$, \textit{assuming} that a result similar to theorem \ref{compoL} is true. This is for instance the case of the Radon transform -see above- and the X-Ray transform -proved by Taryn Flock, to appear.  More accurately, we want to prove the following:
\begin{theorem}\label{uniqi}
Let $1 \leq k \leq d-1$. Assume that any extremizer for the $k$-plane transform inequality \eqref{map1} can be written $f \circ L$ with $f$ a radial, nonincreasing extremizer and $L$ an affine map. Then any extremizer can be written
\begin{equation}\label{uniqextr}
\left[\dfrac{C}{1+\| L x \|^2}\right]^\frac{k+1}{2}
\end{equation}
with $C>0$ and $L$ an invertible affine map.
\end{theorem}

This is not as simple as we could guess. Indeed, it shows that we can deduce global uniqueness from the single uniqueness modulo radial extremizers. Of course one of the main tool here will be the use of the symmetry $\Ss$ combined with the fact that an extremizer is a radial function composed with an affine map. Thus we will use again the competing symmetry theory. From now we will assume that $k$ is such that any extremizer for \eqref{map1} can be written $f \circ L$ with $f$ radial and $L$ an affine map. Our main lemma is the following:

\begin{lem}
Let $f$ be a radial nonincreasing extremizer. Then $\left(V\Ss\right)^2$ acts on $f$ as a dilatation.
\end{lem}

Let us choose $f$ a radial nonincreasing extremizer. Then $f$ is not the -almost everywhere- null function. Thus there exists $\lambda_0 > 0$ such that $f\left(\lambda_0 e_d\right) \neq 0$. As we will explain later, there is no harm assuming $\lambda_0 =1$, using that the dilations group is a symmetry group.\\

$\Ss f $ is also an extremizer. It follows that there exist $F:\R^+ \rightarrow \R$, nonincreasing, a linear invertible map $L$ and a vector $x_0 \in \R^d$ such that 
\begin{equation}\label{equat199}
\Ss f \left(x\right) = F\left( \| x_0 + Lx \| \right).
\end{equation} 
Computing $\Ss f \left( u , \sqrt{u^2+1} \right)$, we get
\begin{equation}\label{equ17}
f\left( e_d \right) \left[\dfrac{1}{1+u^2} \right]^\frac{k+1}{2} = F\left(\| x_0 + Lu + \sqrt{1+u^2} L e_d \|\right),
\end{equation}
for all $u \in \R^{d-1} \times \left\{ 0 \right\}$. Let $C = f\left(e_d\right) \neq 0$.

\begin{lem}
The map $F$ is decreasing, as a function of the norm.
\end{lem}

\begin{proof}
Let us assume that there exists $0 \leq \alpha \leq \beta$ such that $F$ is constant on $\left[\alpha, \beta\right]$. If $F\left(\alpha\right)=0$ then since $F$ is nonnegative, nonincreasing, $F=0$ on $\left[\alpha, \infty\right)$. Thus, because of \eqref{equ17}, for all $u \in \Rd$,
\begin{equation*}
\| Lu + \sqrt{1+u^2} Le_d +x_0 \| < \alpha.
\end{equation*}
This is not possible: indeed, let us write $Lz_0=x_0$, then since $L^{-1}$ is Lipschitz, there exists a constant $c>0$ such that for all $u \in \Rd$
\begin{equation}\label{eq56}
c^2 \| u + \sqrt{1+u^2} e_d +z_0 \|^2 < \alpha^2 .
\end{equation}
Let us develop the right member:
\begin{equation}\label{eqa}
\| u + \sqrt{1+u^2} e_d +z_0 \|^2 = 1+2u^2 + 2\lr{u,z_0} + 2\lr{\sqrt{1+u^2}e_d, z_0}+2\lr{\sqrt{1+u^2}e_d,u}.
\end{equation}
Now let us choose $u=re_1$. Then \eqref{eqa} is equivalent to $2r^2$ as $r \rightarrow \infty$. Thus \eqref{eq56} cannot hold and $F\left(\alpha\right)$ must be positive.\\

 Then it follows from \eqref{equ17} that whenever $\|x_0 + Lu + \sqrt{1+u^2} L e_d \| \in \left[\alpha, \beta\right]$, $\| u \|$ must be constant, let us say equal to $R$:
\begin{equation*}
\|x_0 + Lu + \sqrt{1+R^2} L e_d \| \in \left[\alpha, \beta\right] \Rightarrow \| u \| =R.
\end{equation*}
Let us call $y_0 = x_0 + \sqrt{1+R^2} L e_d$ and $v=Lu+y_0$. Then
\begin{equation*}
\|v \| \in \left[\alpha, \beta\right] \Rightarrow \| L^{-1}\left(v-y_0\right) \| =R.
\end{equation*}
Let us call $\CC \subset \R^{d-1}$ the ring of minimum radius $\alpha$ and maximum radius $\beta$, and $H$ the $d-1$-plane $L^{-1}\left(\Rd-y_0\right)$. The application
$$\begin{array}{ccccc}
& \Psi: & \CC & \longrightarrow & H\\
& & v & \longmapsto &  L^{-1}\left(v-y_0\right) \end{array}$$
is Lipschitz, and so is its inverse. Thus 
\begin{equation*}
\HHH ^{d-1}\left(\Psi\left(\CC\right)\right) \geq c \HHH ^{d-1}\left(\CC\right)
\end{equation*}
where $\HHH^{d-1}$ is the $d-1$-Hausdorff measure and $c>0$ is a constant. But $\Psi\left(\CC\right)$ is the intersection of the $d-1$-plane $H$ and the $d-1$-sphere of radius $R$, which is of null $\HHH^{d-1}$-measure, thus so is $\CC$: $\alpha = \beta$.\end{proof}

The function $F$ is then injective. That shows that $\|x_0 + Lu + \sqrt{1+u^2} L e_d \|$ must be a function of $u^2$ only. The following lemma makes us conclude:

\begin{lem}
Let $L$ an invertible linear map such that $\|x_0 + Lu + \sqrt{1+u^2} L e_d \|$ depends only on $\| u \|$. Then $L\left(\Rd\right) \subset \spane\left(Le_d\right)^\perp$, and $L|_{\Rd}$ preserves the norm, modulo a multiplicative constant. At last, there exists $s_0 \in \R^d$ such that $x_0 = s_0 Le_d$.
\end{lem}

\begin{proof}
Let us choose $u= \theta \in S^{d-2}\times \left\{ 0 \right\}$. Then 
\begin{equation*}
\|x_0 + Lu + \sqrt{1+u^2} L e_d \|^2 = \|L\theta \|^2 + \| \sqrt{2}Le_d +x_0 \|^2 + 2\lr{L\theta , \sqrt{2}Le_d +x_0}
\end{equation*}
is constant, and so does $\|L\theta \|^2 + 2\lr{L\theta , \sqrt{2}Le_d +x_0}$. Let us call $C_0$ its value. Then we have the polynomial equality
$$\|L\theta \|^2 + 2\lr{L\theta , \sqrt{2}Le_d +x_0} = C_0$$
which can hold for degree reasons only if $\|L\theta\|$ is a constant and $\lr{L\theta , \sqrt{2}Le_d +x_0}$ is a constant. Here we must assume $d \geq 3$, so the sphere $S^{d-2}$ contains an infinity of points.\\

The condition $\|L\theta\|$ constant holds only if $L|_{\Rd}$ preserves the norm, modulo a multiplicative constant. Thus coming back to the assumption of the lemma, with $u=r\theta$ for all $r \geq 0$, the quantity
$$\lr{L\theta , \sqrt{1+r^2}Le_d +x_0}$$
must depend only on $r$. Using $\theta$ and $-\theta$, for all $r$, $\lr{L\theta , \sqrt{1+r^2}Le_d +x_0}=0$. But since $L$ is invertible, the space spanned by $L\theta$ has dimension $d-1$. Thus the space spanned by the vectors $\sqrt{1+r^2}Le_d +x_0$ for $r\geq 0$ has dimension $1$, which proves that there exists $s_0$ such that $s_0 Le_d = x_0$.\end{proof}

Composing with an isometry we can assume that $L\left(\Rd\right) \subset \Rd$. Moreover $\| Lu \|$ depends only on $\|u\|$, which implies that $L$ restrained to $\Rd$ must be a multiple of an isometry. We then deduce that there exist $a,b,s_0$ such that $\| L\left(u + s e_d\right) +x_0 \|^2 = a^2 u^2 + b^2 \left(s+s_0\right)^2$, \textit{for all} $\left(u,s\right) \in \R^{d-1} \times \R$. 
Thus we get the fundamental relation between $f$ and $F$:
\begin{equation*}
\Ss f\left(u+se_d\right) = F\left(\sqrt{a^2 u^2 + b^2 \left(s+s_0\right)^2}\right).
\end{equation*}

In the case $\lambda_0 \neq 1$, let us call $g$ the function
\begin{equation*}
g\left(x\right) = f\left(\dfrac{x}{\lambda_0}\right).
\end{equation*}
Then $g$ satisfies $g\left(e_d\right) \neq 0$ and there exist $L,x_0$ satisfying \eqref{equat199}. It follows from the above study that there exist $a,b,s_0$ three real numbers such that for all $u \in \Rd$, $s \in \R$,
\begin{equation*}
\Ss g\left(u+se_d\right) = F\left(\sqrt{a^2 u^2 + b^2 \left(s+s_0\right)^2}\right).
\end{equation*}
But $\Ss f$ and $\Ss g$ are linked through
\begin{equation*}
\Ss g \left(u,s\right) = \lambda_0^{k+1} \Ss f \left(u,\lambda_0 s\right).
\end{equation*}
Thus changing $b$ to $\lambda_0 b$ and $s_0$ to $\frac{s_0}{\lambda_0}$, we have the same conclusion.\\

Now changing $F$ to $G=F\left(\sqrt{ab} \cdot\right)$, G remains nonincreasing and we get
\begin{equation*}
\Ss f\left(u+se_d\right) = F\left(\sqrt{a^2 u^2 + b^2 \left(s+s_0\right)^2}\right) = G\left(\sqrt{\dfrac{a}{b}u^2 + \dfrac{b}{a}\left(s+s_0\right)^2} \right),
\end{equation*}
reducing the number of unknown parameters in our system. The following lemma sums up the situation:

\begin{lem}
Let $f$ be a radial nonincreasing extremizer for \eqref{map1}. Then there exist a nonincreasing, function $G$, $s_0 \in \R$ and $c$ a positive number, such that for all $u+se_d \in \R^d$,
\begin{equation*}
\Ss f\left(u+se_d\right) = G\left(\sqrt{c u^2 + \dfrac{1}{c} \left(s+s_0\right)^2}\right).
\end{equation*}
\end{lem}
Thus we accomplished our first step in our identification program: we know how the operator $\Ss$ acts on radial extremizers. Now we have to understand how $V$ acts on functions $g$ whose form is
\begin{equation*}
g: u+se_d \mapsto G\left(\sqrt{c u^2 + \dfrac{1}{c} \left(s+s_0\right)^2}\right).
\end{equation*}
This is way easier. First, we can assume that $s_0 = 0$: indeed, $g\left(\cdot - s_0 e_d\right)^\ast = g^\ast$. Moreover, $G$ is decreasing and so the level sets of $g$ are ellipsoids $c u^2 + c^{-1} s^2 \leq R^2$. The corresponding rearranged sets are balls of radius $R'$, with $R'$ satisfying the relation
$$ R'^d = \dfrac{R^{d-1}}{c^\frac{d-1}{2}} c^\frac{1}{2} R =  \dfrac{R^d}{c^\frac{d-2}{2}}. $$
 Thus  
\begin{equation*}
Vg \left(s e_d\right) = G\left(c^\frac{d-2}{2d} s \right) = \dfrac{1}{\left(c^\frac{d-1}{d} s - s_0\right)^{k+1}} f\left( \dfrac{e_d}{c^\frac{d-1}{d} s - s_0} \right),
\end{equation*}
coming back to the relation defining $G$, and using that $f$ is radial. And then
\begin{equation*}
Vg \left( \left(s+c^\frac{d}{d-1} s_0 \right) e_d\right) = \dfrac{1}{\left(c^\frac{d-1}{d} s\right)^{k+1}} f\left( \dfrac{e_d}{c^\frac{d-1}{d} s} \right).
\end{equation*}
The right function even, and so is $Vg$. This forces $s_0$ to be equal to $0$.\\

This characterizes the action of the operator $V\Ss$ on radial extremizers. Indeed, calling $\lambda = c^\frac{d-1}{d}$, we get the following lemma:

\begin{lem}
Let $f$ be a radial nonincreasing extremizer. Then there exists $\lambda$ such that
\begin{equation*}
V\Ss f \left(x\right) = \dfrac{1}{\lambda^{k+1} \| x \|^{k+1}} f \left(\dfrac{e_d}{\lambda \|x \|} \right).
\end{equation*}
\end{lem}
Let us use again the competing symmetry theory: to construct an explicit extremizer of \eqref{map2} we used iterations of $V\Ss$, applied to \textit{any} extremizer. Let us choose $f_0$ a radial extremizer. From now we will regard all the radial functions as functions of the norm instead of functions on $\R^d$. Then $V \Ss f_0$ is still a radial extremizer and because of the previous lemma we know that there exists $\lambda$ such that 
\begin{equation*}
V\Ss f \left(r\right) = \left(\dfrac{1}{\lambda r}\right)^{k+1} f \left(\dfrac{1}{\lambda r} \right).
\end{equation*}
Let us do that again: there exists $\lambda'$ such that
\begin{align*}
\left(V\Ss\right)^2 f \left(r\right) & = \left(\dfrac{1}{\lambda' r}\right)^{k+1} \left(V\Ss f\right) \left(\dfrac{1}{\lambda' r} \right) \\
               & = \left(\dfrac{1}{\lambda' r} \dfrac{\lambda' r}{\lambda}\right)^{k+1} f \left( \dfrac{\lambda' r}{\lambda} \right)\\
               & = \dfrac{1}{\lambda^{k+1}} f \left( \dfrac{\lambda' r}{\lambda} \right).
\end{align*}
Since the operator $V \Ss$ preserves the norm, we must have $\lambda \lambda'^d =1$. With the parameter $\mu$ such that $\lambda'=\mu\lambda$ we get the following lemma:

\begin{lem}
Let $f$ be a radial, nonincreasing extremizer for \eqref{map1}. Then there exists a real number $\mu>0$ such that 
\begin{equation*}
\left(V \Ss\right)^2 f\left(r\right) = \mu^\frac{d}{p} f\left(\mu r\right).
\end{equation*}
\end{lem}
That proves that the operator $V\Ss$ acts on radial, nonincreasing extremizers as a dilatation. Now we are almost done. Indeed, let us consider $f_n = \left(V \Ss\right)^{2n} f$. For each $n$, there exists $\mu_n$ such that
\begin{equation*}
\left(V \Ss\right)^{2n} f\left(r\right) = \mu_n^\frac{d}{p} f\left(\mu_n r\right).
\end{equation*}
But the sequence $f_n$ converges in $L^p$ to the extremizer $h$ described in theorem \ref{sharp}. Thus it converges weakly to a non-zero function, which is possible if and only if $\mu_n$ converges to a non-zero value. That ends the proof of \ref{uniqi}: every nonnegative radial extremizer can be written 
\begin{equation*}
x \mapsto \left[ \dfrac{1}{a+b \|x\|^2} \right]^\frac{k+1}{2}
\end{equation*}
with $a,b > 0$.

\end{document}